\documentclass[12pt]{amsart}

\usepackage{mathtools}

\usepackage{amssymb}
\usepackage{hyperref, xcolor}
\usepackage[margin=.9in,includefoot]{geometry}
\usepackage{paralist, comment}

\newtheorem{theorem}{Theorem}[section]

\newtheorem{proposition}[theorem]{Proposition}

\newtheorem*{recap f}{Proposition~\ref{prop: computing f}}

\theoremstyle{definition}

\newcommand{\F}{\mathbf{F}}
\newcommand{\Fq}{\F_q}
\newcommand{\x}{\mathbf{x}}
\newcommand{\y}{\mathbf{y}}

\newcommand{\GL}{\mathrm{GL}}

\newcommand{\CC}{\mathbb{C}}
\newcommand{\ZZ}{\mathbb{Z}}
\newcommand{\Ima}{\operatorname{Im}}

\newcommand{\normchi}{\widetilde{\chi}}

\newcommand{\Cusp}{\textrm{Cusp}}
\newcommand{\Par}{\textrm{Par}}

\newcommand\Irr{\operatorname{Irr}}

\newcommand\Tr{\operatorname{Tr}}
\newcommand\hook[2]{{\left\langle #2-#1, 1^{#1} \right\rangle}}
\newcommand\basis[2]{\frac{(#1; q^{-1})_{#2}}{(q; q)_{#2}}}
\newcommand\llambda{{\underline{\lambda}}}
\newcommand\CCC{{\operatorname{Cusp}}}
\newcommand\1{{\bf 1}}
\newcommand\defn[1]{{\bf #1}}
\newcommand{\C}{\mathcal{C}}
\newcommand{\wt}{\operatorname{wt}}
\newcommand{\cycles}{\operatorname{cycles}}

\theoremstyle{remark}
\newtheorem{remark}[theorem]{Remark}
\newtheorem{question}[theorem]{Question}

\numberwithin{equation}{section}
\newcommand\qbin[3]{\left[\begin{matrix} #1 \\ #2 \end{matrix} \right]_{#3}}

\def\SS{\mathfrak{S}}

\begin{document}

\title{$\GL_n(\F_q)$-analogues of factorization problems in the symmetric group}
\author{Joel Brewster Lewis}
\address{J.~B. Lewis \\ University of Minnesota, Twin Cities}
\email{jblewis@math.umn.edu}
\author{Alejandro H. Morales}
\address{A.~H. Morales \\ University of California, Los Angeles}
\email{ahmorales@math.ucla.edu}
\date{\today}

\begin{abstract}
We consider $\GL_n(\F_q)$-analogues of certain factorization
problems in the symmetric group $\SS_n$: rather than counting factorizations of
the long cycle $(1,2,\ldots,n)$ given the
number of cycles of each factor, we count factorizations of a 
regular elliptic element given the fixed space dimension of each 
factor. We show that, as in $\SS_n$, the generating function counting these
factorizations has attractive coefficients after
an appropriate change of basis.  Our work generalizes several recent
results on factorizations in $\GL_n(\F_q)$ and also uses a
character-based approach.

As an application of our results, we compute the asymptotic growth rate of the number of factorizations of fixed genus of a regular elliptic element in $\GL_n(\F_q)$ into two factors as $n \to \infty$.  We end with a number of open questions.
\end{abstract}

\maketitle

\tableofcontents

\newpage

\section{Introduction}

There is a rich vein in combinatorics of problems related to
factorizations in the symmetric group~$\SS_n$.  Frequently, the size of
a certain family of factorizations is unwieldy but has an
attractive generating function, possibly after an appropriate change
of basis. As a prototypical example, one might seek to count
factorizations $c = u \cdot v$ of the \defn{long cycle} $c = (1, 2,
\ldots, n)$ in $\SS_n$ as a product of two permutations, keeping track
of the \emph{number of cycles} or even the \emph{cycle types} of the
two factors.  Such results have been given by Harer--Zagier \cite[\S
5]{HarerZagier} when one of the factors is a fixed point-free involution, and in
the general setting by Jackson \cite[\S4]{Jackson1},
\cite{Jackson2}.

\begin{theorem}[Jackson \cite{Jackson2}; Morales--Vassilieva \cite{MV}]
\label{q=1 two factors}

Let $a_{r, s}$ be the number of pairs $(u, v)$ of elements of $\SS_n$
such that $u$ has $r$ cycles, $v$ has $s$ cycles, and $c = u \cdot
v$. Then
\begin{equation} \label{eq:Sn-twofactors-cycles}
\frac{1}{n!} \sum_{r, s \geq 0} a_{r, s} \cdot x^r y^s
=
\sum_{t, u \geq 1} 
          \binom{n - 1}{t - 1; u - 1; n - t - u + 1} \binom{x}{t} \binom{y}{u}.
\end{equation}
Moreover, for $\lambda, \mu$ partitions of $n$, let $a_{\lambda, \mu}$ be the number of pairs $(u, v)$ of elements of $\SS_n$ such that $u$ has cycle type $\lambda$, $v$ has cycle type $\mu$, and $c = u \cdot v$.  Then
\begin{equation}
\label{q=1 refined generating function}
\frac{1}{n!}\sum_{\lambda, \mu \vdash n} a_{\lambda, \mu} \cdot p_\lambda(\x) p_\mu(\y)
= 
\sum_{\alpha,\beta}\frac{(n-\ell(\alpha))!(n-\ell(\beta))!}{(n - 1)! (n+1-\ell(\alpha)-\ell(\beta))!}
{\x}^{\alpha} {\y}^{\beta},
\end{equation}
where $p_\lambda$ denotes the usual power-sum symmetric function and the sum on the right is over all weak compositions $\alpha, \beta$ of $n$.
\end{theorem}

Recently, there has been interest in $q$-analogues of such problems,
replacing $\SS_n$ with the finite general linear group $\GL_n(\Fq)$,
the long cycle with a \emph{Singer cycle} (or, more generally,
\emph{regular elliptic element}) $c$, and the number of cycles with
the \emph{fixed space dimension} \cite{LRS, HLR}; or in more general
geometric settings \cite{HaLeR-V}.  In the present paper, we extend this approach to give the following $q$-analogue of Theorem~\ref{q=1 two factors}.  Our theorem statement uses the standard notations
\[
(a; q)_m  = (1 - a)(1 - aq) \cdots (1 - aq^{m - 1}) 
\]
and
\[
[m]!_q = \frac{(q; q)_m}{(1 - q)^m} = 1\cdot (1 + q) \cdots (1 + q + \ldots + q^{m - 1}).
\]
\begin{theorem}
\label{thm: two factors}
Fix a regular elliptic element $c$ in $G = \GL_n(\Fq)$.  Let $a_{r, s}(q)$ be the number of pairs $(u, v)$ of elements of $G$ such that $u$ has fixed space dimension $r$, $v$ has fixed space dimension~$s$, 
and $c = u \cdot v$.  Then
\begin{multline}
\label{eq: two factor theorem}
\frac{1}{|G|} \sum_{r, s \geq 0} a_{r, s}(q) \cdot x^r y^s
 = 
\frac{(x;q^{-1})_n}{(q;q)_n} + \frac{(y;q^{-1})_n}{(q;q)_n} 
+ {} \\
\sum_{\substack{0\leq t,u \leq n-1\\ t+u\leq n}}
q^{tu-t-u}
\frac{[n-t-1]!_q \cdot [n-u-1]!_q}{[n - 1]!_q \cdot [n-t-u]!_q} \frac{(q^n -q^{t}-q^{u}+1)}{(q-1)} 
\cdot
\frac{(x;q^{-1})_{t}}{(q;q)_{t}} \frac{(y;q^{-1})_{u}}{(q;q)_{u}}.
\end{multline}
\end{theorem}

More generally, in either $\SS_n$ or $\GL_n(\F_q)$ one may consider factorizations into more than two factors.  In $\SS_n$, this gives the following result. 
\begin{theorem}[Jackson \cite{Jackson2}; Bernardi--Morales \cite{BernardiM}]
\label{q=1 many factors}
Let $a_{r_1,r_2,\ldots,r_k}$ be the number of $k$-tuples
$(u_1,u_2,\ldots,u_k)$ of permutations in $\SS_n$ such that
$u_i$ has $r_i$ cycles and $u_1u_2\cdots u_k =c$. Then 
\begin{equation}
\label{eq:Jackson's formula into many factors}
\frac{1}{(n!)^{k - 1}}\sum_{1 \leq r_1,r_2,\ldots,r_k \leq n} a_{r_1,\ldots,r_k} \cdot x_1^{r_1}\cdots x_k^{r_k} = 
\sum_{1\leq p_1,\ldots,p_k \leq n}
M^{n-1}_{p_1-1,\ldots,p_k-1} \binom{x_1}{p_1}\cdots \binom{x_k}{p_k},
\end{equation}
where 
\begin{align} 
\label{eq:defM}
M^{m}_{r_1,\ldots,r_k} &:= 
\sum_{d=0}^{\min(r_i)} (-1)^d \binom{m}{d} \prod_{i=1}^k \binom{m-d}{r_i-d}.
\end{align}
Moreover, let $a_{\lambda^{(1)},\ldots,\lambda^{(k)}}$ be the number of $k$-tuples
$(u_1,u_2,\ldots,u_k)$ of permutations in $\SS_n$ such that
$u_i$ has cycle type $\lambda^{(i)}$ and $u_1u_2\cdots u_k =c$. Then
\[
\frac{1}{(n!)^{k - 1}}
\sum_{\lambda^{(1)},\cdots,\lambda^{(k)} \vdash n} 
\hspace{-.25in}
a_{\lambda^{(1)},\ldots,\lambda^{(k)}} \cdot
p_{\lambda^{(1)}}({\x}_1)\cdots p_{\lambda^{(k)}}(\x_k) 
= \hspace{-.1in}
\sum_{\alpha^{(1)},\ldots,\alpha^{(k)}}
\hspace{-.1in}
\frac{M^{n-1}_{\ell(\alpha^{(1)})-1,\ldots,\ell(\alpha^{(k)})-1}}{\prod_{i=1}^k
 \binom{n-1}{\ell(\alpha^{(i)})-1}}({\x}_1)^{\alpha^{(1)}} \cdots ({\x}_k)^{\alpha^{(k)}}, 
\]
where the sum on the right is over all weak compositions $\alpha^{(1)},\ldots,\alpha^{(k)}$ of $n$.
\end{theorem}

In the present paper, we prove the following $q$-analogue of this
result.  The statement uses the standard $q$-binomial coefficient ${
  \footnotesize \qbin{n}{k}{q}} = [n]!_q {\big /} \left([k]!_q \cdot [n - k]!_q\right)$.
\begin{theorem}
\label{thm: many factors}
Fix a regular elliptic element $c$ in $G = \GL_n(\Fq)$.  Let $a_{r_1, \ldots, r_k}(q)$ be the number of tuples $(u_1, \ldots, u_k)$ of elements of $G$ such that $u_i$ has fixed space dimension $r_i$ and $u_1 \cdots u_k = c$.  Then 
\begin{equation} \label{eq:genseries-manyfactors} 
\frac{1}{|G|^{k - 1}}
\sum_{r_1, \ldots, r_k} a_{r_1, \ldots, r_k}(q) \cdot x_1^{r_1} \cdots x_k^{r_k}
=
\sum_{ \substack{{\bf p} = (p_1, \ldots, p_k) \colon \\ 0 \leq p_i \leq n}  } 
\frac{M^{n - 1}_{\widetilde{\bf p}}(q)}{\prod_{p \in \widetilde{\bf p}} \qbin{n-1}{p}{q}} \cdot 
\frac{(x_{1};q^{-1})_{p_1}}{(q;q)_{p_1}} \cdots \frac{(x_{k};q^{-1})_{p_k}}{(q;q)_{p_k}},
\end{equation}
where $\widetilde{\bf p}$ is the result of deleting all copies of $n$ from $\bf p$,
\begin{equation} \label{eq:def-qM}
M^{m}_{r_1,\ldots,r_k}(q) := \sum_{d=0}^{\min(r_i)} (-1)^d
q^{\binom{d+1}{2}-kd} \qbin{m}{d}{q} \prod_{i=1}^k \qbin{m-d}{r_i-d}{q}
\end{equation}
for $k > 0$, and $M^{m}_{\varnothing}(q) := 0$.
\end{theorem}

\begin{remark}
In viewing Theorems~\ref{thm: two factors} and~\ref{thm: many factors} as 
$q$-analogues of Theorems~\ref{q=1 two factors} and~\ref{q=1 many factors}, 
it is helpful to first observe that if $x = q^N$ is a
positive integer power of $q$ then 
$\dfrac{(x; q^{-1})_m}{(q; q)_m} = {\footnotesize \qbin{N}{m}{q}}$. 
Further, we have the equality
\[
\lim_{q\to 1} \frac{[n-t-1]!_q \cdot [n-u-1]!_q}{[n - 1]!_q \cdot [n-t-u]!_q} \frac{(q^n -q^{t}-q^{u}+1)}{(q-1)} = \frac{(n-(t+1))!(n-(u+1))!}{(n - 1)! (n+1-(t+1)-(u+1))!}
\]
between the limit of a coefficient 
in~\eqref{eq: two factor theorem} and a coefficient on 
the right side of~\eqref{q=1 refined generating function}, 
and more generally the equality $\lim_{q \to 1}M^m_{r_1, \ldots, r_k}(q) = M^m_{r_1, \ldots, r_k}$.  

Note that the generating function \eqref{eq:genseries-manyfactors} is (in its definition) analogous to the less-refined generating function \eqref{eq:Jackson's formula into many factors}, while the coefficient 
\[
M^{n - 1}_{\widetilde{\bf p}}(q) \Big/ \prod_{p \in \widetilde{\bf p}}\qbin{n - 1}{p}{q}
\]
is analogous (in the $q \to 1$ sense) to a coefficient in the more
refined half of Theorem~\ref{q=1 many factors}.  This phenomenon is
mysterious.  A similar phenomenon was observed in the discussion following Theorem~4.2 in \cite{HLR}, namely, that the counting formula $q^{e(\alpha)}(q^n - 1)^{k - 1}$ for factorizations of a regular elliptic element in $\GL_n(\F_q)$ into $k$ factors with fixed space codimensions given by the composition $\alpha$ of $n$ is a $q$-analogue of the counting formula $n^{k - 1}$ for factorizations of an $n$-cycle as a genus-$0$ product of $k$ \emph{cycles} of specified lengths.

On the other hand, we \emph{can} give a heuristic explanation for the fact that the lower indices in the $M$-coefficients in Theorem~\ref{thm: many factors} are shifted by $1$ compared with those in Theorem~\ref{q=1 many factors}: the matrix group $\SS_n$ does not act irreducibly in its standard representation, as every permutation fixes the all-ones vector.  Thus, morally, the subtraction of $1$ should correct for the irrelevant dimension of fixed space.
\end{remark}

Our approach is to follow a well-worn path, based on 
character-theoretic techniques that go back to Frobenius.  In the case
of the symmetric group, this approach has been extensively developed
in the '80s and '90s, notably in work of Stanley~\cite{Sta81},
Jackson~\cite{Jackson1, Jackson2},
Hanlon--Stanley--Stembridge~\cite{HSS}, and
Goupil--Schaeffer~\cite{GoupilSchaeffer}; see also the survey \cite{GouldenJacksonSurvey}. In $\GL_n(\Fq)$, the necessary character theory was worked out by Green \cite{Green}, building on work of Steinberg \cite{Steinberg}.  This approach has been used recently by the first-named author and coauthors to count factorizations of Singer cycles into reflections \cite{LRS} and to count \emph{genus-$0$} factorizations (that is, those in which the codimensions of the fixed spaces of the factors sum to the codimension of the fixed space of the product) of regular elliptic elements \cite{HLR}.  The current work subsumes these previous results while requiring no new character values.  (See Remarks~\ref{rmk: recovering LRS} and~\ref{rmk: genus 0 many factors} for derivations of these earlier results from Theorem~\ref{thm: many factors}.)

The plan of the paper is as follows.  In Section~\ref{method of
  attack}, we provide background, including an overview of the
character-theoretic approach to problems of this sort and a quick
introduction to the character theory of $\GL_n(\Fq)$ necessary for our
purposes.  Theorems~\ref{thm: two factors} and~\ref{thm: many factors}
are proved in Sections~\ref{section: two factors} and~\ref{section:
  many factors}, respectively.  The \defn{genus} of a factorization
counted in $a_{r,s}(q)$ is $n-r-s$. In Section~\ref{section: asymptotics},
we give an application of Theorem~\ref{thm: two factors} to asymptotic
enumeration, giving the precise growth rate 
$\Theta(q^{(n + g)^2/2}/|\GL_g(\Fq)|)$ of the number of factorizations of
fixed genus $g$ of a regular elliptic element in $\GL_n(\Fq)$ as a
product of two factors, as $n \to \infty$.  Finally, in
Section~\ref{section: conclusion} we give a few additional remarks and
open questions.  In particular, in Section~\ref{sec:combproofs} we
briefly discuss the history of \emph{combinatorial} approaches to
Theorems~\ref{q=1 two factors} and~\ref{q=1 many factors}, and
discuss whether this story can be given a $q$-analogy.

\subsection*{Acknowledgements}

We are grateful to Olivier Bernardi, Valentin F\'eray, and Vic Reiner for helpful comments and suggestions during the preparation of this paper.  We are deeply indebted to Dennis Stanton for his help and guidance in dealing with $q$-series manipulations at the core of Sections~\ref{section: two factors} and~\ref{section: many factors}.  JBL was supported by NSF grant DMS-1401792.  

\section{Regular elliptics, character theory, and the symmetric group approach}
\label{method of attack}

\subsection{Singer cycles and regular elliptic elements}
The field $\F_{q^n}$ is an $n$-dimensional vector space over $\Fq$, and multiplication by a fixed element in the larger field is a linear transformation.  Thus, any choice of basis for $\F_{q^n}$ over $\Fq$ gives a natural inclusion $\F_{q^n}^\times \hookrightarrow G_n := \GL_n(\Fq)$.  The image of any cyclic generator $c$ for $\F_{q^n}^\times$ under this inclusion is called a \defn{Singer cycle}.  A strong analogy between Singer cycles in $G_n$ and $n$-cycles in $\SS_n$ has been established over the past decade or so, notably in work of Reiner, Stanton, and collaborators \cite{CSP, ReinerStantonWebb, LRS, HLR}.  As one elementary example of this analogy, the Singer cycles act transitively on the lines of $\Fq^n$, just as the $n$-cycles act transitively on the points $\{1, \ldots, n\}$.

A more general class of elements of $G_n$, containing the Singer
cycles, is the set of images (under the same inclusion) of field generators~$\sigma$ for $\F_{q^n}$ over $\Fq$.  (That is, one should have $\Fq[\sigma] = \F_{q^n}$ but not necessarily $\{\sigma^m \mid m \in \ZZ\} = \F_{q^n}^\times$.)  Such elements are called \defn{regular elliptic elements}.  They may be characterized in several other ways; see \cite[Prop.~4.4]{LRS}.

\subsection{The character-theoretic approach to factorization problems}
\label{character approach in S_n}
Given a finite group $G$, let 
$\Irr(G)$ be the collection of its irreducible (finite-dimensional,
complex) representations $V$.  For each $V$ in $\Irr(G)$, denote by
$\deg(V) := \dim_\CC V$ its \defn{degree}, 
by $\chi^V(g) := \Tr(g: V \rightarrow V)$ its \defn{character value} at $g$,
and by $\normchi^V(g):=\chi^V(g) / \deg(V)$ its
\defn{normalized character value}. The functions $\chi^V(-)$ and $\normchi^V(-)$ on $G$
extend by $\CC$-linearity to functionals on the group algebra $\CC[G]$.  

The following result allows one to express every factorization problem of the form we consider as a computation in terms of group characters.

\begin{proposition}[Frobenius \cite{frobenius}]
\label{general factorization prop}
Let $G$ be a finite group, and $A_1,\ldots,A_\ell \subseteq G$
unions of conjugacy classes in $G$.  Then for $g$ in $G$,
the number of ordered factorizations $(t_1,\ldots,t_\ell)$ with
$g=t_1 \cdots t_\ell$ and $t_i$ in $A_i$ for $i=1,2,\ldots,\ell$ is
\begin{equation}
\label{frobenius factorization equation}
\frac{1}{|G|} \sum_{V \in \Irr(G)} 
 \deg(V) \chi^V(g^{-1}) \cdot \normchi^V(z_1) \cdots \normchi^V(z_\ell),
\end{equation}
where $z_i:=\sum_{t \in A_i} t$ in $\CC[G]$.
\end{proposition}

In practice, it is often the case that one does not need the full set of character values that appear in \eqref{frobenius factorization equation} in order to evaluate the sum.  As an example of this phenomenon, we show how to derive Theorem~\ref{q=1 two factors} without needing access to the full character table for the symmetric group $\SS_n$.  This argument also provides a template for our work in $\GL_n(\Fq)$.

Let $c$ be the long cycle $c = (1, 2, \ldots, n)$ in $\SS_n$.  Consider the generating function 
\begin{equation}
\label{definition of the gf in S_n}
F(x, y) = \sum_{1 \leq r, s \leq n} a_{r, s} \cdot x^r y^s
\end{equation}
for the number $a_{r, s}$ of factorizations $c = u \cdot v$ in which $u$, $v$ have $r$, $s$ cycles, respectively.  By Proposition~\ref{general factorization prop}, we have that
\[
a_{r, s} = \frac{1}{n!} \sum_{V \in \Irr(\SS_n)} 
 \deg(V) \chi^V(c^{-1}) \cdot \normchi^V(z_r) \normchi^V(z_s),
\]
where $z_i$ is the formal sum in $\CC[\SS_n]$ of all elements with $i$ cycles.  Substituting this in \eqref{definition of the gf in S_n} gives
\begin{align*}
F(x, y) & = \frac{1}{n!} \sum_{1 \leq r, s \leq n} \sum_{V \in \Irr(\SS_n)} 
 \deg(V) \chi^V(c^{-1}) \cdot \normchi^V(z_r) \normchi^V(z_s) \cdot x^r y^s \\
       & = \frac{1}{n!} \sum_{V \in \Irr(\SS_n)} \deg(V) \chi^V(c^{-1}) \cdot f_V(x) f_V(y),
\end{align*}
where 
$
f_V(x) := \sum_{r = 1}^n \normchi^V(z_r) x^r.
$
The irreducible representations of $\SS_n$ are indexed by partitions $\lambda$ of $n$, and we write $\chi^V = \chi^\lambda$ if $V$ is indexed by $\lambda$.  The degree of a character is given by the \emph{hook-length formula}.  It follows from the \emph{Murnaghan--Nakayama rule} that the character value $\chi^{\lambda}(c^{-1})$ on the $n$-cycle $c^{-1}$ is equal to $0$ unless $\lambda = \hook{d}{n}$ is a \defn{hook shape}, in which case $\chi^{\hook{d}{n}}(c^{-1}) = (-1)^d$.  Thus it suffices to understand $f_\lambda(x)$ for hooks $\lambda$.  One can show that  
\[
f_{\hook{d}{n}}(x) = (x - d)_n := (x - d) \cdot (x - d + 1) \cdot (x - d + 2) \cdots (x - d + n - 1) = n! \cdot \sum_{k = d + 1}^n \binom{n - 1 - d}{k - 1 - d} \cdot \binom{x}{k}
,
\]
and the result follows by
identities for binomial coefficients
after
extracting the coefficient of $\binom{x}{t}
\binom{y}{u}$.

\subsection{Character theory of the finite general linear group}

In this section, we give a (very) brief overview of the character theory of $G_n = \GL_n(\Fq)$, including the specific character values necessary to prove the main results in this paper.  For a proper treatment, see \cite[Ch.~3]{Zelevinsky} or \cite[\S4]{GrinbergReiner}.  Throughout this section, we freely conflate the (complex, finite-dimensional) representation $V$ for $\GL_n(\Fq)$ with its character $\chi^V$.

The basic building-block of the character theory of $G_n$ is \defn{parabolic} (or \defn{Harish-Chandra}) \defn{induction}, defined as follows.  For nonnegative integers $a, b$, let $P_{a, b}$ be the parabolic subgroup
\[
P_{a, b} = \left \{ \begin{bmatrix} A & C \\ 0 & B \end{bmatrix} \colon
A \in G_a, B \in G_b, \textrm{ and } C \in \Fq^{a \times b} \right\}
\]
of $G_{a + b}$.  Given characters $\chi_1$ and $\chi_2$ for $G_a$ and $G_b$, respectively, one obtains a character $\chi_1 * \chi_2$ for $G_{a + b}$ by the formula
\[
(\chi_1 * \chi_2)(g) = \frac{1}{|P_{a, b}|} \sum_{\substack{h \in G_{a + b} \colon \\hgh^{-1} \in P_{a, b}}} \chi_1(A)\chi_2(B),
\]
where $A$ and $B$ are the diagonal blocks of $hgh^{-1}$ as above.

Many irreducible characters for $G_n$ may be obtained as irreducible
components of induction products of characters on smaller general
linear groups.  A character $\C$ for $G_n$ that cannot be so-obtained
is called \defn{cuspidal}, of \defn{weight} $\wt(\C) = n$.  The set of
cuspidals for $G_n$ is denoted $\Cusp_n$, and the set of all cuspidals
for all general linear groups is denoted $\Cusp = \sqcup_{n \geq 1}
\Cusp_n$.  (Though we will not need this, we note that cuspidals may
be indexed by irreducible polynomials over $\Fq$, or equivalently by
primitive $q$-colored necklaces.)

Let $\Par$ denote the set of all integer partitions.  The set of \emph{all} irreducible characters for $G_n$ is indexed by functions
$
\llambda \colon \Cusp \to \Par
$
such that
\[
\sum_{\C \in \Cusp} \wt(\C) \cdot |\llambda(\C)| = n.
\]
(In particular, $\llambda(\C)$ must be equal to the empty partition for all but finitely many choices of $\C$.)  A particular representation of interest is the trivial representation $\1$ for $\GL_1(\Fq)$.  (The trivial representation for $\GL_n(\Fq)$ is indexed by the function associating to $\1$ the partition $\langle n \rangle$.)  If $V$ is indexed by $\llambda$ having support on a single cuspidal representation $\C$, we call $V$ \defn{primary} and denote it by the pair $(\C, \lambda)$ where $\lambda = \llambda(\C)$.

\emph{A priori}, in order to use Proposition~\ref{general factorization prop} in our setting, we require the degrees and certain other values of all irreducible characters for $G_n$.  In fact, however, we will only need a very small selection of them.  The character degrees were worked out by Steinberg~\cite{Steinberg} and Green~\cite{Green}, and the special case relevant to our work is
\begin{equation}
\label{hook-degree-formula}
\deg\left( \chi^{\1, \hook{d}{n}} \right)=
q^{\binom{d+1}{2}} \qbin{n-1}{d}{q}.
\end{equation}
The relevant character values on regular elliptic elements were computed by Lewis--Reiner--Stanton.
\begin{proposition}[{\cite[Prop.~4.7]{LRS}}]
\label{Singer-cycle-character-values}
Suppose $c$ is a regular elliptic element and $\chi^{\llambda}$ an irreducible character of $G_n$.
\begin{compactenum}[(i)]
\item One has $
\chi^{\llambda}(c) = 0
$
unless $\chi^{\llambda}$ is a primary irreducible character $\chi^{U,\lambda}$
for some $s$ dividing $n$ and some cuspidal character $U$ in $\CCC_s$, and 
$\lambda = \hook{d}{\frac{n}{s}}$ is a hook-shaped partition of $n/s$.  
\item If $U = \1$ is the trivial character then 
\[
\chi^{\1, \hook{d}{n}}(c) = (-1)^d.
\]
\end{compactenum}
\end{proposition}

Finally, denote by $z_k$ the formal sum (in $\CC[G_n]$) of all elements of $G_n$ having fixed space dimension\footnote{Caution: our indexing here differs from that of \cite{HLR}, where the symbol $z_k$ is used to represent the sum of elements with fixed space \emph{co}dimension $k$.} equal to $k$.  Huang--Lewis--Reiner computed the relevant character values on the $z_k$.
\begin{proposition}[{\cite[Prop.~4.10]{HLR}}]
\label{character values prop}
\begin{compactenum}[(i)]
\item For any $s$ dividing $n$, any cuspidal representation $U$ in $\CCC_s$ other than $\1$, and any partition $\lambda$ of $\frac{n}{s}$, we have
\[
\normchi^{U, \lambda}(z_r) = (-1)^{n - r} q^{\binom{n - r}{2}} \qbin{n}{r}{q}.
\]
\item For $U = \1$ and $\lambda = \hook{d}{n}$ a hook, we have
\begin{multline*}
\normchi^{\1, \hook{d}{n}}(z_r) = 
(-1)^{n - r} q^{\binom{n - r}{2}} \Big(\qbin{n}{r}{q} + \\
\frac{(1-q)[n]_q}{[r]!_q} \cdot \sum_{j = 1}^{n - \max(r, d)} q^{jr - d} \cdot \frac{[n - j]!_q }{[n - r - j]!_q}\cdot(q^{n-d-j+1};q)_{j-1}\Big).
\end{multline*}
\end{compactenum}
\end{proposition}

\section{Factoring regular elliptic elements into two factors}
\label{section: two factors}
\subsection{Proof of Theorem~\ref{thm: two factors}}
In this section, we prove Theorem~\ref{thm: two factors} by following the approach for $\SS_n$ sketched in Section~\ref{character approach in S_n}.  Let $c$ be a regular elliptic element in $G = \GL_n(\Fq)$, and let $a_{r, s}(q)$ be the number of pairs $(u, v)$ of elements of $G$ such that $u \cdot v = c$ and $u$, $v$ have fixed space dimensions $r$, $s$, resprectively.
Define the generating function
\[
F(x, y) := \sum_{r, s\geq 0} a_{r, s}(q) x^r y^s.
\]
Our goal is to rewrite this generating function in the basis $\frac{(x; q^{-1})_t}{(q; q)_t}\frac{(y; q^{-1})_u}{(q; q)_u}$ of polynomials $q$-analogous to the binomial coefficients.  

By Proposition~\ref{general factorization prop}, we may write
\begin{equation}
\label{factorization count}
a_{r, s}(q) = \frac{1}{|G|}\sum_{V \in \Irr(G)} \deg(V) \chi^V(c^{-1}) \cdot \normchi^V(z_{r}) \cdot \normchi^V(z_{s})
\end{equation}
where $z_k$ is defined (as above) to be the element of the group algebra $\CC[G]$ equal to the sum of all elements of fixed space dimension $k$.
Thus, our generating function is given by
\begin{align}
F(x, y) 
& = \sum_{r, s\geq 0} x^r y^s  \frac{1}{|G|}\sum_{V \in \Irr(G)} \deg(V) \chi^V(c^{-1}) \cdot \normchi^V(z_{r}) \cdot \normchi^V(z_{s})    \notag\\
& = \frac{1}{|G|}\sum_{V \in \Irr(G)} \deg(V) \chi^V(c^{-1}) \cdot \left(\sum_r  \normchi^V(z_{r}) x^r\right) \cdot  \left(\sum_s  \normchi^V(z_{s}) y^s\right)    \notag\\
\label{F in terms of f}
& =\frac{1}{|G|} \sum_{V \in \Irr(G)} \deg(V) \chi^V(c^{-1}) \cdot f_V(x) \cdot f_V(y),
\end{align}
where 
\[
f_V(x) := \sum_{r = 0}^n \normchi^V(z_r) \cdot x^{r}.
\]

By Proposition~\ref{Singer-cycle-character-values}, the character value $\chi^V(c^{-1})$ is typically $0$, and so in order to prove Theorem~\ref{thm: two factors} it suffices to compute $f_V$ for only a few select choices of $V$.  We do this now.

\begin{proposition}
\label{prop: computing f}
If $V = (U, \lambda)$ for $U \neq \1$ we have
\begin{equation}
\label{easy f}
f_{U, \lambda}(x) = |G| \cdot \frac{(x; q^{-1})_n}{(q; q)_n},
\end{equation}
while if $V = (\1, \hook{d}{n})$ we have
\begin{equation}
\label{expressed in new basis}
f_{\1, \hook{d}{n}}(x) = |G| \cdot  \left(\frac{(x; q^{-1})_n}{(q; q)_n} + q^{-d} \cdot \sum_{m = d}^{n - 1} \frac{[m]!_q \cdot [n - d - 1]!_q}{[m - d]!_q \cdot [n - 1]!_q} \cdot \frac{(x; q^{-1})_m}{(q; q)_m}\right).
\end{equation}
\end{proposition}
The proof is a reasonably straightforward computation using Proposition~\ref{character values prop}, the $q$-binomial theorem (essentially \cite[(1.87)]{ec1})
\begin{equation}
\label{q-binomial 1}
\frac{(x; q^{-1})_m}{(q; q)_m} = \frac{1}{(q; q)_m q^{\binom{m}{2}}} \sum_{k = 0}^m (-1)^k q^{\binom{m - k}{2}} \qbin{m}{k}{q} \cdot x^k,
\end{equation}
and its inverse
\begin{equation}
\label{q-binomial 2}
x^k = \sum_{m = 0}^k (-1)^m q^{\binom{m}{2}} (q^k; q^{-1})_m \cdot \frac{(x; q^{-1})_m}{(q; q)_m}.
\end{equation}
It is hidden away in Appendix~\ref{calculation of f}.

We continue studying the factorization generating function $F(x, y)$.  We split the sum \eqref{F in terms of f} according to whether $V$ is a primary irreducible over the cuspidal $\1$ of hook shape:
\begin{multline}
\label{manipulating F}
|G| \cdot F(x, y)  = \sum_{V \neq (\1, \hook{d}{n})} \deg(V) \chi^V(c^{-1}) \cdot f_V(x) \cdot f_V(y) + {} \\
 \sum_{d = 0}^{n - 1} \deg(\1, \hook{d}{n}) \chi^{\1, \hook{d}{n}}(c^{-1}) \cdot f_{\1, \hook{d}{n}}(x) \cdot f_{\1, \hook{d}{n}}(y).
\end{multline}
We use Proposition~\ref{Singer-cycle-character-values}(i) and \eqref{easy f} to rewrite the first sum on the right side of \eqref{manipulating F} as
\begin{align*}
\sum_{V \neq (\1, \hook{d}{n})} 
\hspace{-.25in}
\deg(V) 
\chi^V(c^{-1}) \cdot 
f_V(x) 
f_V(y) 
& = 
 \sum_{V = (U, \lambda), U \neq \1} \deg(V) \chi^V(c^{-1}) \cdot f_V(x) f_V(y) \\
& =
\sum_{V = (U, \lambda), U \neq \1} \deg(V) \chi^V(c^{-1}) \cdot |G|^2 \cdot \basis{x}{n} \basis{y}{n} \\
& = |G|^2 \basis{x}{n} \basis{y}{n} \sum_{V \neq (\1, \hook{d}{n})} \deg(V) \chi^V(c^{-1}).
\end{align*}
Observe (following the same idea as in \cite[\S4.3]{HLR}) that $\sum_{V \in \Irr(G)} \deg(V) \chi^V$ is the character of the regular representation for $G$.  It follows that $\sum_{V \in \Irr(G)} \deg(V) \chi^V(c^{-1}) = 0$ and so that
\begin{equation}
\label{simple part of F}
\begin{aligned}
\sum_{V \neq (\1, \hook{d}{n})}\hspace{-.25in} \deg(V) 
\chi^V(c^{-1}) \cdot 
f_V(x) 
f_V(y) 
& = - |G|^2 \basis{x}{n} \basis{y}{n} \sum_{V = (\1, \hook{d}{n})} \deg(V) \chi^V(c^{-1}) \\
& = - |G|^2 \basis{x}{n} \basis{y}{n} \cdot \sum_{d = 0}^{n - 1} q^{\binom{d + 1}{2}} \qbin{n - 1}{d}{q} \cdot (-1)^d \\
&= - (q; q)_{n - 1} \cdot |G|^2 \cdot \basis{x}{n} \basis{y}{n}.
\end{aligned}
\end{equation}

Substituting from \eqref{hook-degree-formula}, \eqref{simple part of F} and Proposition~\ref{Singer-cycle-character-values} into \eqref{manipulating F} yields
\begin{multline}
\label{two factors simplified}
\frac{F(x, y)}{|G|} =  - (q; q)_{n - 1} \basis{x}{n} \basis{y}{n} + {}\\
\frac{1}{|G|^2}\sum_{d = 0}^{n - 1} (-1)^d q^{\binom{d + 1}{2}}
\qbin{n - 1}{d}{q} \cdot f_{\1, \hook{d}{n}}(x) \cdot f_{\1, \hook{d}{n}}(y).
\end{multline}
In order to finish the proof of Theorem~\ref{thm: two factors}, we must extract the coefficient of $\basis{x}{t} \basis{y}{u}$ from the right side of this equation.  Call this coefficient $b_{t, u}(q)$, so that 
\[
F(x, y) = |G| \cdot \sum_{t, u}  b_{t, u}(q) \cdot \basis{x}{t} \basis{y}{u}.
\]
By \eqref{expressed in new basis}, this extraction reduces to a computation involving $q$-series.  
We begin with a few small simplifications in order to make the computation more reasonable.

First, observe that the subadditivity of fixed space codimensions \cite[Prop.~2.9]{HLR} implies that $a_{r, s}(q) = 0$ if $r + s > n$.  By the triangularity of the change-of-basis formulas \eqref{q-binomial 1}, \eqref{q-binomial 2}, it follows immediately that $b_{t, u}(q) = 0$ if $t + u > n$.  When $r + s = n$, the value of $a_{r, s}(q)$ was computed (by the same methods) in \cite{HLR}, yielding $a_{n, 0}(q) = a_{0, n}(q) = 1$ and $a_{r, s}(q) = q^{2rs - n}(q^n - 1)$ if $r, s$ are both positive.  Again by the triangularity of \eqref{q-binomial 2} we have 
\[
b_{r, s}(q) = a_{r,s}(q) \cdot (-1)^r q^{\binom{r}{2}} (q^r; q^{-1})_r \cdot (-1)^s q^{\binom{s}{2}} (q^s; q^{-1})_s \big\slash |G| 
\]
when $r + s = n$.  This yields $b_{n, 0}(q) = b_{0, n}(q) = 1$ and $b_{r, s}(q) = q^{rs - r - s}(q^n - 1) \frac{[r]!_q [s]!_q}{[n]!_q}$ when $r + s = n$, and this latter formula may be seen by a short computation to be equal to the desired value.  Thus, it remains to compute $b_{t, u}(q)$ when $0 \leq t, u $ and $t + u < n$.  

In this case, it follows from \eqref{two factors simplified} and \eqref{expressed in new basis} that
\begin{equation}
\label{interesting coefficient}
b_{t, u}(q) = \sum_{d = 0}^{\min(t, u)} (-1)^d q^{\binom{d + 1}{2}}\qbin{n - 1}{d}{q} \cdot
q^{-d}\frac{[t]!_q \cdot [n - d - 1]!_q}{[t - d]!_q \cdot [n - 1]!_q} 
\cdot 
q^{-d}\frac{[u]!_q \cdot [n - d - 1]!_q}{[u - d]!_q \cdot [n - 1]!_q}.
\end{equation}
Collecting powers of $q$, cancelling common factors and rewriting with the easy identity
\begin{equation}
\label{falling factorial to pochhammer}
\frac{[m]!_q}{[m - d]!_q} = \frac{q^{md - \binom{d}{2}}}{(q - 1)^d} \cdot (q^{-m}; q)_d
\end{equation} 
shows that $b_{t, u}(q)$ is equal to the $q$-hypergeometric function\footnote{
Recall the definition
\[
_r\phi_s(a_1, \ldots, a_r; \quad b_1, \ldots, b_s; \quad z) := \sum_{d = 0}^\infty \frac{(a_1; q)_d \cdot (a_2; q)_d \cdots (a_r; q)_d}{(q; q)_d \cdot (b_1; q)_d \cdots (b_s; q)_d} \left((-1)^d q^{\binom{d}{2}}\right)^{1 + s - r} z^d.
\]
}
\[
b_{t, u}(q) = {_2\phi_1}(q^{-t}, \; q^{-u}; \quad q^{1-n}; \quad q^{t + u - n}).
\]
This function is almost summable by the $q$-Chu--Vandermonde identity (see \cite[(II.7)]{GasperRahman}), but the power of $q$ in the final parameter is off by one; instead, we use the $_2\phi_1$-to-$_3\phi_2$ identity \cite[(III.7)]{GasperRahman}
\[
_2\phi_1(q^{-k}, \; B; \quad C; \quad z)= \frac{(C/B; q)_k}{(C; q)_k} \cdot {_3\phi_2}(q^{-k}, \; B, \; Bzq^{-k}/C; \quad Bq^{1-k}/C,\; 0; \quad q)
\]
to rewrite
\[
b_{t, u}(q) = \frac{(q^{1 + u - n}; q)_t}{(q^{1-n}; q)_t} {_3\phi_2}(q^{-t}, \; q^{-u}, \; q^{-1}; \quad q^{n - t- u}, \; 0; \quad q).
\]
Since one of the numerator parameters is $q^{-1}$, this $_3\phi_2$ summation truncates after two terms:
\begin{align*}
b_{t, u}(q) & =  \frac{(q^{1 + u - n}; q)_t}{(q^{1-n}; q)_t} \sum_{d = 0}^\infty \frac{(q^{-t}; q)_d (q^{-u}; q)_d (q^{-1}; q)_d}{(q; q)_d (q^{n - t - u}; q)_d} q^{d} \\
& = 
\frac{(q^{1 - n + u}; q)_{t}}{(q^{1-n}; q)_{t}} \cdot \left(1   +    \frac{ (1-q^{-t}) \cdot (1-q^{-u}) \cdot (1-q^{-1}) }{(1 - q) \cdot (1-q^{n-t-u})} \cdot q \right) \\
& = 
     q^{tu} \frac{[n - t - 1]!_q \cdot [n - u - 1]!_q}{[n - t - u - 1]!_q \cdot [n - 1]!_q} \cdot \left(1   -    \frac{ (1-q^{-t}) \cdot (1-q^{-u}) }{ 1-q^{n-t-u}} \right) \\
& =      q^{tu - t - u} \frac{[n - t - 1]!_q \cdot [n - u - 1]!_q}{[n - t - u]!_q \cdot [n - 1]!_q} \cdot \frac{q^{n} - q^{t}-q^{u} + 1}{q - 1}.
\end{align*}
This concludes the proof of Theorem~\ref{thm: two factors}.

\subsection{Additional remarks on Theorem~\ref{thm: two factors}}
We note here two special cases of Theorem~\ref{thm: two factors}; compare with the remarks following \cite[Thm.~1.1]{SchaefferVassilieva}.
\begin{remark}
\label{rmk: total number of factorizations}
If we set $x = y = 1$ in \eqref{eq: two factor theorem}, all terms on the right vanish except $t = u = 0$, leaving
\[
F(1, 1) = |\GL_n(\Fq)| \cdot \frac{[n - 1]!_q \cdot [n - 1]!_q}{[n - 1]!_q \cdot [n]!_q} \cdot \frac{q^n - 1}{q - 1}
= |\GL_n(\Fq)|.
\]
Indeed, for each element $u$ of $\GL_n(\Fq)$ there is exactly one element $v$ such that $u \cdot v = c$.
\end{remark}
\begin{remark}
More generally, if we set $y = 1$ in \eqref{eq: two factor theorem}, we get the generating function $F(x, 1)$ for $\GL_n(\Fq)$ by dimension of fixed space.  On the right side, only terms with $u = 0$ survive, and so
\begin{align*}
F(x, 1) & = |\GL_n(\Fq)| \cdot \left(\frac{(x;q^{-1})_{n}}{(q;q)_{n}} + 
\sum_{t=0}^{n - 1}
q^{-t} \frac{[n-t-1]!_q \cdot [n-1]!_q}{[n - 1]!_q \cdot [n-t]!_q} \cdot \frac{q^n -q^{t}}{q - 1}
\frac{(x;q^{-1})_{t}}{(q;q)_{t}} \right)\\
& = |\GL_n(\Fq)| \cdot \sum_{t=0}^n
\frac{(x;q^{-1})_{t}}{(q;q)_{t}}.
\end{align*}
This formula is analogous to the generating function 
for $\SS_n$ by number of cycles; see also \cite[\S3]{HLR}.
Formulas for the coefficients of this generating function were 
obtained by Fulman \cite{Fulman} (attributed there 
to unpublished work of Rudvalis and Shinoda).  
\end{remark}

\section{Factoring regular elliptic elements into more than two factors}
\label{section: many factors}
\subsection{Proof of Theorem~\ref{thm: many factors}}
We follow the same framework as in Section~\ref{section: two factors}.  Let $a_{r_1, \ldots, r_k}(q)$ be the number of tuples $(g_1, \ldots, g_k)$ of elements of $G = \GL_n(\Fq)$ such that $g_i$ has fixed space dimension $r_i$ for all $i$ and $g_1 \cdots g_k = c$.  Define the generating function
\[
F(x_1, \ldots, x_k) = \sum_{r_1, \ldots, r_k} a_{r_1, \ldots, r_k}(q) x_1^{r_1} \cdots x_k^{r_k}.
\]
The statement we wish to prove asserts a formula for this generating function when expressed in another basis.
Applying Proposition~\ref{general factorization prop} and making calculations analogous to those in \eqref{F in terms of f} gives
\[
|G| \cdot F(\x) = \sum_{V \in \Irr(G)} \deg(V) \chi^V(c^{-1}) f_V(x_1) \cdots f_V(x_k).
\]
The same regular representation trick that leads from \eqref{manipulating F} to \eqref{two factors simplified} works with more variables; it yields
\begin{multline}
\label{many factors simplified}
|G| \cdot F(\x) = -(q; q)_{n - 1} |G|^k \cdot \frac{(x_1; q^{-1})_n}{(q; q)_n} \cdots \frac{(x_k; q^{-1})_n}{(q; q)_n} + {}\\
\sum_{d = 0}^{n - 1} (-1)^d q^{\binom{d + 1}{2}} \qbin{n - 1}{d}{q} f_{\1, \hook{d}{n}}(x_1) \cdots f_{\1, \hook{d}{n}}(x_k).
\end{multline}
To finish the proof of Theorem~\ref{thm: many factors}, we must extract from this expression the coefficient of $\prod_i \basis{x_i}{p_i}$.  For convenience, we denote by $b_{p_1, \ldots, p_k}(q)$ the coefficient of $\prod_i \basis{x_i}{p_i}$ in the normalized generating function $F(\x)/|G|^{k - 1}$.

Because of the form \eqref{expressed in new basis} of the polynomial $f_{\1, \hook{d}{n}}(x)$, it is convenient to introduce a new parameter $j$, marking the number of indices $p_i$ not equal to $n$.  Up to permuting variables, we may assume without loss of generality that $p_i < n$ if $1 \leq i \leq j$ and $p_i = n$ if $j < i \leq k$.  Then substituting from \eqref{expressed in new basis} into \eqref{many factors simplified} and extracting coefficients yields
\begin{align}
b_{p_1, \ldots, p_k}(q)  & = \sum_{d = 0}^{\min(p_i)} (-1)^d
q^{\binom{d + 1}{2}} \qbin{n - 1}{d}{q} q^{-jd} \left(\frac{[n - d -
    1]!_q}{[n - 1]!_q}\right)^j \prod_{i=1}^j \frac{[p_i]!_q}{[p_i -
  d]!_q} \notag \\
& = \sum_{d=0}^{\min(p_i)}
 (-1)^d q^{\binom{d+1}{2}-jd} \qbin{n - 1}{d}{q} \prod_{i=1}^j \qbin{n - 1 - d}{p_i - d}{q} \Big\slash
{\prod_{i=1}^j \qbin{n-1}{p_i}{q}}, \label{interesting coefficient
  many factors}
\end{align}
as desired.  This completes the proof of Theorem~\ref{thm: many factors}.

\begin{remark}
When $k=2$ the expression in \eqref{interesting coefficient many factors} 
equals \eqref{interesting coefficient}. In the case of two factors we 
applied an additional $q$-identity to arrive at the expression in
Theorem~\ref{thm: two factors}. In the case of $k$ factors, using a
trick of Stanton (personal communication), one can split 
\eqref{interesting coefficient many factors} into several terms and 
apply similar identities (like \cite[(3.2.2)]{GasperRahman} for 
the case $k=3$). However, the resulting ``less alternating'' 
expressions have several terms.  Because of this, and like in the 
$\mathfrak{S}_n$ case \eqref{eq:Jackson's formula into many factors}, 
we opted to stop at \eqref{interesting coefficient many factors}.
\end{remark}

\subsection{ Additional remarks on Theorem~\ref{thm: many factors}}
The following remarks show how to recover the main theorems of \cite{LRS, HLR} as special cases of Theorem~\ref{thm: many factors}.  Incidentally, the first remark also settles a conjecture of Lewis--Reiner--Stanton.

\begin{remark}
\label{rmk: recovering LRS}
In \cite[Thm.~1.2]{LRS}, Lewis--Reiner--Stanton gave the following formula for the number $t_q(n, \ell) := a_{n - 1, \ldots, n - 1}(q)$ of factorizations of a Singer cycle $c$ as a product of $\ell$ reflections (that is, elements with fixed space dimension $n - 1$):
\[
t_q(n,\ell) = \frac{(-[n]_q)^{\ell}}{q^{\binom{n}{2}}(q;q)_n} \left(
  (-1)^{n-1}(q;q)_{n-1} + \sum_{k=0}^{n-1}
  (-1)^{k+n}q^{\binom{k+1}{2}} \qbin{n-1}{k}{q} (1+q^{n-k-1}-q^{n-k})^{\ell} \right).
\]
Here, we show how to derive this formula from Theorem~\ref{thm: many factors}.  It was conjectured \cite[Conj.~6.3]{LRS} that this formula should count factorizations of any regular elliptic element, not just a Singer cycle; since the derivation here is valid for all regular elliptic elements, it settles the conjecture.

Using square brackets to denote coefficient extraction, we have by definition that 
\begin{equation} \label{eq:refl2Fseries}
t_q(n,\ell) = [x_1^{n-1}\cdots
x_\ell^{n-1}]\,F(x_1,\ldots,x_{\ell}),
\end{equation} 
where $F(x_1,\ldots,x_\ell)$ is the generating function appearing in \eqref{eq:genseries-manyfactors}. 
By \eqref{q-binomial 1} we have
\[
[x^{n-1}] \frac{(x;q^{-1})_p}{(q;q)_p} 
= 
\begin{cases}
(-1)^{n-1} q^{-\binom{n - 1}{2}}/ (q;q)_{n-1} &\text{ if } p=n-1,\\
(-1)^{n-1} q^{-\binom{n}{2}}[n]_q\,/ (q;q)_n &\text{ if } p=n,\\
0 &\text{ otherwise}. 
\end{cases}
\]
Thus, carrying out the coefficient extraction \eqref{eq:refl2Fseries} from $F(x_1,\ldots,x_\ell)$ yields
\begin{align*}
t_q(n,\ell) &= |G|^{\ell-1} \sum_{i=0}^{\ell} \binom{\ell}{i}
(-1)^{(n-1)\ell} \left( \frac{1}{(q;q)_{n-1} q^{\binom{n-1}{2}}}\right)^i
\left( \frac{[n]_q}{(q;q)_{n} q^{\binom{n}{2}}}  \right)^{\ell-i} M^{n-1}_{(n-1)^i}(q) \\   
&= (-1)^{n} \frac{(-[n]_q)^{\ell}}{q^{\binom{n}{2}}(q;q)_n} \left(
  \sum_{i=0}^{\ell} \binom{\ell}{i} (1-q)^i q^{(n-1)i}  M^{n-1}_{(n-1)^i}(q)    \right).
\end{align*}
By the definition of $M^{m}_{\bf p}(q)$ 
we have $M^{n-1}_{\varnothing}=0$ and $M^{n-1}_{(n-1)^i}(q)  = \sum_{k=0}^{n-1} (-1)^k q^{\binom{k+1}{2}- ik}
{\footnotesize \qbin{n-1}{k}{q}}$ for $i>0$.
Thus
\begin{align*}
t_q(n,\ell) &= (-1)^{n} \frac{(-[n]_q)^{\ell}}{q^{\binom{n}{2}}(q;q)_n} \left(
  \sum_{i=1}^{\ell} \binom{\ell}{i} (1-q)^i q^{(n-1)i}
  \left(   \sum_{k=0}^{n-1} (-1)^k q^{\binom{k+1}{2}- ik}
\qbin{n-1}{k}{q} \right)    \right)\\
&= (-1)^{n} \frac{(-[n]_q)^{\ell}}{q^{\binom{n}{2}}(q;q)_n} \left(
\sum_{k=0}^{n-1} (-1)^k q^{\binom{k+1}{2}} \qbin{n-1}{k}{q} \left(
  \sum_{i=1}^{\ell} \binom{\ell}{i} q^{(n-k-1)i}(1-q)^i
\right) \right)\\
&=  \frac{(-[n]_q)^{\ell}}{q^{\binom{n}{2}}(q;q)_n} \left(
\sum_{k=0}^{n-1} (-1)^{k+n} q^{\binom{k+1}{2}} \qbin{n-1}{k}{q} \left((1+q^{n-k-1}-q^{n-k})^{\ell}-1\right)
\right),
\end{align*}
where the last step follows by an application of the binomial
theorem. Finally, by \eqref{q-binomial 1} we have
\begin{align*}
\sum_{k=0}^{n-1} (-1)^{k} q^{\binom{k+1}{2}} \qbin{n-1}{k}{q} = (q;q)_{n-1}
\end{align*}
and so
\begin{align*}
t_q(n,\ell) &= \frac{(-[n]_q)^{\ell}}{q^{\binom{n}{2}}(q;q)_n} \left(
\sum_{k=0}^{n-1} (-1)^{k+n} q^{\binom{k+1}{2}} \qbin{n-1}{k}{q} \left((1+q^{n-k-1}-q^{n-k})^{\ell}-1\right)
\right)\\
&= \frac{(-[n]_q)^{\ell}}{q^{\binom{n}{2}}(q;q)_n} \left(
(-1)^{n-1}(q;q)_{n-1} + \sum_{k=0}^{n-1} (-1)^{k+n} q^{\binom{k+1}{2}} \qbin{n-1}{k}{q} (1+q^{n-k-1}-q^{n-k})^{\ell}\right),
\end{align*}
as desired.
\end{remark}

\begin{remark}
\label{rmk: genus 0 many factors}
In the genus-$0$ case $r_1 + \ldots + r_k = (k - 1)n$, there is a simple formula \cite[Thm.~4.2]{HLR} 
\begin{equation}
\label{eq:genus 0}
a_{r_1, \ldots, r_k}(q) = q^{\sum_{i = 1}^k (n - r_i - 1)r_i} (q^n - 1)^{k - 1}
\end{equation}
for $a_{r_1, \ldots, r_k}(q)$ with $0 \leq r_i < n$.  Here, we show how to derive this formula from Theorem~\ref{thm: many factors}.

In the genus-$0$ case, it follows from the triangularity of the 
basis change \eqref{q-binomial 1} 
that $a_{r_1, \ldots, r_k}(q)$ is equal to a predictable factor times 
$b_{r_1, \ldots, r_k}(q)$.  Thus, we focus on computing this
latter coefficient.  Our approach is to evaluate a certain $q$-difference 
in two different ways: the first will give a formula involving
$M^{m}_{\bf r}(q)$, while the second will give an explicit product formula.

Let $\Delta_q$ be the operator that acts on a function $f$ by 
\[
\Delta_q(f)(x) := \frac{f(qx)-f(x)}{qx - x}
  =\frac{f(qx)-f(x)}{(q-1) x}.
\]  
Then it is not hard to show using the $q$-Pascal recurrence \cite[(1.67)]{ec1} that 
\begin{equation}
\label{iterated q difference}
 \Delta_q^{N}(f)(x) = q^{-\binom{N}{2}} (q - 1)^{-N} 
    \sum_{d= 0}^{N}(-1)^{d} q^{\binom{d}{2}} \qbin{N}{d}{q} \frac{ f(q^{N - d}x) }{x^N}.
\end{equation}
On the other hand, $\Delta_q$ acts in a predictable way on polynomials: if $f$ is a polynomial of degree $n$ and leading coefficient $a$ then $\Delta_q(f)$ is a polynomial of degree $n - 1$ and leading coefficient $[n]_q \cdot a$.  

If one defines for $0 \leq r \leq m$ the polynomial 
\[
P^m_{r}(x) := x \cdot \frac{(x q^{r + 1 - m}; q)_{m - r}}{(q; q)_{m - r}} 
= x \cdot \frac{(x; q^{-1})_{m - r}}{(q; q)_{m - r}}
= x \cdot(1 - x)(1 - xq^{-1}) \cdots (1 - xq^{m - r - 1})/(q; q)_{m - r}
\]
and for a tuple ${\bf r} = (r_1, \ldots, r_k)$
\[
\mathcal{P}^m_{\bf r}(x) := \frac{1}{x} \prod_{\substack{1 \leq i \leq k \colon \\ r_i \leq m}} P^m_{r_i}(x),
\]
then one has by Equation~\eqref{iterated q difference} that
\begin{align*}
 \Delta_q^{m}(\mathcal{P}^m_{{\bf r}})(x) & = q^{-\binom{m}{2}} (q - 1)^{-m} 
    \sum_{d= 0}^{m}(-1)^{d} q^{\binom{d}{2}} \qbin{m}{d}{q} \frac{ \mathcal{P}^m_{{\bf r}}(q^{m - d}x) }{x^m} \\
& = q^{-\binom{m}{2}} (q - 1)^{-m} 
    \sum_{d= 0}^{m}(-1)^{d} q^{\binom{d}{2}} \qbin{m}{d}{q} \frac{(q^{m - d}x)^{k - 1}}{x^m} \prod_{i = 1}^k \frac{(xq^{m - d}; q^{-1})_{m - r_i}}{(q; q)_{m - r_i}} \\
& = q^{-\binom{m}{2} + m(k - 1)} (q - 1)^{-m} 
    \sum_{d= 0}^{m}(-1)^{d} q^{\binom{d}{2} - d(k - 1)} \qbin{m}{d}{q} x^{k - m - 1} \prod_{i = 1}^k \frac{(xq^{m - d}; q^{-1})_{m - r_i}}{(q; q)_{m - r_i}}.
\end{align*}
Plugging in $x = 1$ gives 
\begin{align}
\Delta_q^{m}(\mathcal{P}^m_{{\bf r}})(1) & =
q^{-\binom{m}{2} + m(k - 1)} (q - 1)^{-m}
\sum_{d= 0}^{m}(-1)^{d} q^{\binom{d}{2} - d(k - 1)} \qbin{m}{d}{q} \prod_{i = 1}^k \frac{(q^{m - d}; q^{-1})_{m - r_i}}{(q; q)_{m - r_i}} \notag\\
& =
q^{-\binom{m}{2} + m(k - 1)} (q - 1)^{-m}
\sum_{d= 0}^{m}(-1)^{d} q^{\binom{d}{2} - d(k - 1)} \qbin{m}{d}{q} \prod_{i = 1}^k \qbin{m - d}{m - r_i}{q} \notag\\
\label{difference formula for M}
& =
q^{-\binom{m}{2} + m(k - 1)} (q - 1)^{-m} M^{m}_{\bf r}(q).
\end{align}

On the other hand, when $m = n - 1$ we have that $\mathcal{P}^m_{\bf r}$ is a polynomial of degree $(k - 1) + \sum_i (n - 1 - r_i) = kn - 1 - \sum_i r_i$.  Thus, in the genus-$0$ case $\sum_i r_i = (k - 1)n$, we have that $\deg \mathcal{P}^m_{\bf r} = n - 1$.  It follows immediately that $\Delta_q^{n - 1}(\mathcal{P}^{n - 1}_{\bf r})$ is equal to 
$[n - 1]!_q$ times the leading coefficient of $\mathcal{P}^{n - 1}_{\bf r}$.  It is easy to see from the definition that $\mathcal{P}^{n - 1}_{\bf r}$ has leading coefficient 
\[
\prod_{r \in {\bf r} \colon r < n } \frac{(-1)^{m - r} q^{-\binom{m - r}{2}}}{(q; q)_{n - 1 - r}}.
\]
Combining this with \eqref{difference formula for M} gives a simple 
product formula\footnote{
More generally, one could obtain by the same techniques a formula for $b_{r_1, \ldots, r_k}(q)$ in the genus $g := (k - 1)n - \sum_i r_i$ case as a sum of $g + 1$ terms, depending on the top $g + 1$ coefficients of $\mathcal{P}$.
}
 for $b_{r_1, \ldots, r_k}(q)$.  Putting everything together gives~\eqref{eq:genus 0}, as desired.
\end{remark}

\begin{remark}
We again suppose that exactly $j$ of the $p_i$ are less than $n$, and without loss of generality take them to be $p_1, \ldots, p_j$.
Applying \eqref{falling factorial to pochhammer},
one may rewrite 
\[
b_{p_1, \ldots, p_k}(q) = b_{p_1, \ldots, p_j}(q)
\] 
as a $q$-hypergeometric function of a particularly simple form:
\[
b_{p_1, \ldots, p_k}(q)
= 
{_j\phi_{j - 1}}(\quad q^{- p_1}, \ldots, q^{-p_j} \quad ; \quad q^{1 - n}, \ldots, q^{1 - n} \quad ; \quad q^{n(1 - j) + \sum p_i}  )
.
\]
Then the calculation in Remark~\ref{rmk: genus 0 many factors} in the genus-$0$ case is equivalent to a $q$-analogue of the Karlsson--Minton formulas; see \cite[\S1.9]{GasperRahman}.
\end{remark}

\section{An application to aysmptotic enumeration of factorizations by genus}
\label{section: asymptotics}
In the spirit of (e.g.)\ \cite{FulmanStanton} 
and \cite[\S4.2]{GoupilSchaeffer}, one may ask to study 
the \emph{asymptotic} enumeration of factorizations.  Here, 
we compute the asymptotic growth of the number of fixed-genus 
factorizations of a regular elliptic element in $G=\GL_n(\Fq)$ 
as $n \to \infty$.
\begin{theorem} \label{thm:growth_rate}
Let $g \geq 0$ and $q$ be fixed.  As $n \to \infty$, the number of genus-$g$ factorizations of a regular elliptic element in $\GL_n(\F_q)$ into two factors has growth rate
\[
\Theta\left(\frac{q^{(n + g)^2/2}}{|\GL_g(\Fq)|} \right),
\]
where the implicit constants depend on $q$ but not $g$.
\end{theorem}

\begin{remark}
\label{rmk: S_n asymptotics}
For the sake of comparison, we include the analogous asymptotics
in $\SS_n$. Goupil and Schaeffer showed 
\cite[Cor.~4.3]{GoupilSchaeffer} that for fixed $g\geq 0$, as $n\to \infty$
the number of genus-$g$ factorizations
of a fixed long cycle in $\SS_n$ into two factors is asymptotic to
\[
\frac{n^{3(g-\frac{1}{2})}4^n}{ g! 48^g\sqrt{\pi}}.
\]
It is interesting to observe that in both results we see that the constant depends on the size ($|\SS_g|$ or $|\GL_g(\Fq)|$) of a related group.  Is there an explanation for this phenomenon?
\end{remark}

The rest of this section is a sketch of the proof of Theorem~\ref{thm:growth_rate}.  To begin, we use Theorem~\ref{thm: two factors} to give an explicit formula for $a_{r, s}(q)$.  For every positive integer $g$, let $P_g(x, y, z, q)$ be the following Laurent polynomial of four variables:
\begin{multline}
\label{eq: polynomial P}
P_g(x, y, z, q) := 
(-1)^g q^{-g} \left( y^{-g} z^{g} \prod_{i = 1}^{g} (yq^{i} - 1) + 
y^{g}z^{-g} \prod_{i = 1}^{g}(zq^{i} - 1) \right) + 
\sum_{\substack{ 0 \leq t', u' \leq g - 1 \\ 0 \leq t' + u' \leq g}} (-1)^{t' + u'} \times {} \\
\qbin{g}{t'; u'; g - t' - u'}{q} \cdot 
y^{u'-t'}z^{t' - u'}
q^{t'u' - t' - u'}
(x-yq^{t'}-zq^{u'}+1) 
\prod_{i = 1}^{g - t' - 1} (zq^{i} - 1)  
\prod_{i = 1}^{g - u' - 1} (yq^{i} - 1).
\end{multline}
\begin{proposition}
\label{lem: explicit formula for a}
If $g, r, s > 0$ satisfy $r + s = n - g$ then 
\[
a_{r, s}(q) =  \frac{q^{2rs + (g-1)n - \binom{g}{2}}(q^n-1)}{(q - 1)^g [g]!_q} \cdot P_g(q^n, q^r, q^s, q).
\]
\end{proposition}
\begin{proof}[Proof sketch.]
Extracting the coefficient of $x^ry^s$ from \eqref{eq: two factor theorem} using \eqref{q-binomial 1} gives
\begin{multline}
\label{eq:qGS} \frac{a_{r,s}(q)}{|G|} = 
(-1)^{n - g} q^{\binom{r+1}{2}+\binom{s+1}{2}} \times {} \\
\sum_{\substack{r \leq t,  \; s \leq u, \\ t + u \leq n}}
\frac{1}{(q;q)_t(q;q)_u} \qbin{t}{r}{q}\qbin{u}{s}{q} \cdot q^{tu-t-u -rt-su}
\frac{[n-t-1]!_q \cdot [n-u-1]!_q}{[n-1]!_q \cdot [n-t-u]!_q} \frac{(q^n-q^t-q^u+1)}{(q-1)}. 
\end{multline}
The rest of the proof is a long but totally unenlightening
calculation: expanding the $q$-binomials, making the change of
variables $t = r + t'$, $u = s + u'$ with $0 \leq t' + u' \leq g$,
separating the $(t', u') = (g, 0)$ and $(0, g)$ terms, rearranging
various factors, and doing some basic arithmetic.  In particular, no
nontrivial $q$-identities are required.
\end{proof}
\begin{proposition}
\label{lem: asymptotics for a sum}
For constants $a, b, c, g$ and $q > 1$, we have
\[
\sum_{\substack{r, s \geq 1 \\ r + s = n - g}} q^{2rs + an + br + cs} 
= 
\Theta\left( q^{(b - c)^2/8 + g(g-b-c)/2} \cdot q^{n^2/2} \cdot q^{(2a + b + c - 2g)n/2} \right)
\]
as $n \to \infty$, where the implicit constants depend only on $q$ and $b - c$.
\end{proposition}
\begin{proof}
In the sum on the left side,
set $m := n - g$ and rearrange to get 
\begin{align*}
\sum_{\substack{r, s \geq 1 \\ r + s = n - g}} q^{2rs + an + br + cs} 
& = q^{an + cm + (m + (b - c)/2)^2/2}\sum_{r = 1}^{m - 1} q^{-\frac{1}{2}((m - 2r) + (b - c)/2)^2}.
\end{align*}
The exponent $an + cm + (m + (b - c)/2)^2/2$ is equal to
$\frac{n^2}{2} + (2a + b + c - 2g)n/2 + \frac{g(g - b - c)}{2} +
\frac{(b - c)^2}{8}$.  For $m$ large, the sum on the right is easily
seen to be $\Theta(1)$: it oscillates with the parity of $m$ between two evaluations of the convergent
Jacobi theta function $\vartheta(w,t):= \sum_{r = -\infty}^\infty t^{r^2} \cdot w^{2r}$, depending on the value $b - c$.  This completes
the proof.
\end{proof}

It follows from the preceding proposition that for any fixed $g$ and any Laurent polynomial $Q(x, y, z)$ with coefficients in $\mathbb{Q}(q)$, the asymptotic growth of $\sum_{r = 1}^{n - g - 1} q^{2r(n - g- r)} Q(q^n, q^r, q^{n - g - r})$ as $n \to \infty$ is determined entirely by the monomials in $Q$ of maximal \defn{weight}, where the weight of $x^a y^b z^c$ is defined to be $2a + b + c$.  This inspires the following result.
\begin{proposition}
\label{lem: maximum weight monomial}
Let $g > 0$ and $P_g(x, y, z, q)$ be as in \eqref{eq: polynomial P}
.  Then $P_g$ has a unique $(x, y, z)$-monomial $x^{1} y^{g - 1} z^{g - 1}$ of maximal weight.
\end{proposition}
\begin{proof}
By inspection, the first two summands in $P_g$ contribute maximum-weight monomials $y^{-g} \cdot z^{g} \cdot y^{g} = x^0y^{0}z^{g}$ and $y^{g} \cdot z^{-g} \cdot z^{g} = x^0y^{g}z^{0}$, both of weight $g/2$.  In the summation, the $t', u'$-summand contributes a unique maximum-weight monomial $z^{g - t' - 1} \cdot y^{g - u' - 1} \cdot y^{u' - t'} \cdot z^{t' - u'} \cdot x = x^1 \cdot y^{g - t' - 1} \cdot z^{g - u' - 1}$, of weight $(2g - t' - u' - 2)/2 + 1$.  This expression is uniquely maximized over the region of summation at $t' = u' = 0$, with value $g$.  This is also strictly larger than the maximum-weight contribution from the first two summands.  Thus, this maximum-weight monomial in the $t' = u' = 0$ summand is the unique maximum-weight monomial in $P_g$, as desired.
\end{proof}

Finally, we put the preceding results together to get asymptotics for factorizations of fixed genus.
\begin{proof}[Proof sketch of Theorem~\ref{thm:growth_rate}.]
The case $g = 0$ is straightforward from the formula~\eqref{eq:genus 0} already available in~\cite{HLR}.  For $g > 0$,  Proposition~\ref{lem: explicit formula for a} provides an explicit polynomial formula for $a_{r, s}(q)$.
Notably, the number of terms in this formula does not depend on $n$.  
 Proposition~\ref{lem: asymptotics for a sum} provides the asymptotics for the sum (over $r, s$ such that $r + s = n - g$) of a single monomial from this polynomial.  
Finally, Proposition~\ref{lem: maximum weight monomial} detects the unique monomial whose asymptotic growth dominates the others.  To finish, one needs to extract the coefficient of the $xy^{g - 1}z^{g - 1}$ monomial in $P_g$ and do the arithmetic.\footnote{We have swept under the rug the issue of the contribution from the terms $a_{0, n - g}(q) = a_{n - g, 0}(q)$.  Using the same approach as in Proposition~\ref{lem: explicit formula for a}, it is another long-but-not-difficult computation to show that this term is low order and so does not contribute to the asymptotic growth rate.}
\end{proof}

\section{Closing remarks}
\label{section: conclusion}

\subsection{Combinatorial proofs} \label{sec:combproofs}
Theorems~\ref{q=1 two factors} and~\ref{q=1 many factors} were
originally proved in \cite{Jackson2} by character methods (as in the outline in Section~\ref{character approach in S_n}), but these are not the only known proofs.  The first result (the case of two factors in $\SS_n$) has several combinatorial proofs,
by Schaeffer--Vassilieva~\cite{SchaefferVassilieva}, 
Bernardi~\cite{Bernardi}, 
and Chapuy--F\'eray--Fusy~\cite{ChapuyFerayFusy}.  
The second result (the case of $k$ factors) has an intricate
combinatorial proof \cite{BernardiM,BernardiM3}.  It would be
of interest to find combinatorial proofs of our $q$-analogous Theorems~\ref{thm: two factors} and~\ref{thm: many factors}

Most of the combinatorial proofs in $\SS_n$ for two and more factors include (following \cite{Bernardi}) the following elements:
\begin{compactenum}[ 1.]
\item For a permutation $\sigma$ in $\mathfrak{S}_n$, denote by
  $\cycles(\sigma)$ the set of cycles of the permutation. When $x,y$ 
  are both positive integers, the term $a_{r,s}x^ry^s$
  counts triples consisting of a factorization $(1,2\ldots,n)=u\cdot v$ and 
  {\em cycle colorings}: functions $\cycles(u)\to \{1,\ldots, x\}$ and
  $\cycles(v) \to \{1,\ldots,y\}$.  
\item When we do the change of bases  $x^k = \sum_{j=0}^k S(k,j)
  (x)_j$ implicit in \eqref{eq:Sn-twofactors-cycles}, where $S(k,j)$ is a
  Stirling number of the first kind and $(x)_j$ is a falling
  factorial, the term
$n!\cdot b_{r,s}$ counts tuples $(u,v,A',B')$ consisting of
\begin{compactitem}
\item  a factorization $(1,2,\ldots,n) = u\cdot v$, 
\item  a surjective function $A': \cycles(u)
  \to \{1,\ldots,r\}$, and
\item a surjective function $B': \cycles(v) \to \{1,\ldots,s\}$,
\end{compactitem}
Call such tuples \defn{colored factorizations}.\footnote{The
  first explicit appearance of cycle coloring of factorizations that we are
  aware of is in 
  \cite[Thm.~2.1]{GouldenJacksonRM}.}
\item Using the classical correspondence (e.g., see
  \cite[Ch.~1]{LandoZvonkin}) between factorizations of a long cycle and {\em unicellular bipartite maps}, certain
  bipartite graphs with a one-cell embedding on a locally orientable
  surface. Under this correspondence, the cycles of $u$ and $v$
  correspond to vertices  
  and colored factorizations
  correspond to unicellular bipartite maps with a coloring on the
  vertices of each part.
\item Doing a bijective construction on the colored bipartite map. (See
  \cite[\S 3]{Bernardi} for the bijective construction proving \eqref{eq:Sn-twofactors-cycles}.)
\end{compactenum} 

The next remark gives a combinatorial interpretation of the coefficients
$|G|\cdot b_{t, u}(q)$ that appear in Theorem~\ref{thm: two factors},
analogous to the first two steps above. However, we do not know if there is a map
interpretation of the factorizations in $\GL_n(\Fq)$ that we consider.

\begin{remark}
The second change of basis formula \eqref{q-binomial 2} may be understood combinatorially in the following way: $x^k$ is the number of linear maps from a $k$-dimensional vector space $V$ to an $x$-element vector space $X$ over $\Fq$.  The term $\basis{x}{m}$ gives the number of $m$-dimensional subspaces of $X$, while $(-1)^m q^{\binom{m}{2}} (q^k; q^{-1})_m = (q^k - 1)(q^k - q)\cdots(q^k - q^{m - 1})$ is the number of surjective maps from $V$ to the selected $m$-dimensional subspace of $X$.  Thus, the right side refines the left side by dimension of the image.

For an element $u$ of $\GL_n(\Fq) \cong \GL(V)$, 
denote by $V^u$ the fixed space $\ker(u - 1)$ of $u$.
When $x, y$ are both powers of $q$, the term $a_{r, s}(q) x^r y^s$ counts
triples consisting of a factorization $c = u \cdot v$ with 
$\dim V^u = r$ and $\dim V^v = s$, a linear map $V^u \to X$, 
and a linear map $V^v \to Y$.  When we change bases, 
the term  $|G|\cdot b_{r,s}(q) \basis{x}{r} \basis{y}{s}$ counts 
triples consisting of a factorization $c = u \cdot v$, 
a linear map $A:V^u\to X$, and a linear map
$B:V^v\to Y$ such that $\dim\Ima(A)=r$ and $\dim\Ima(B)=s$.
Equivalently, $|G|\cdot b_{r, s}(q)$ counts
tuples $(u,v,A',B')$ consisting of
\begin{compactitem}
\item a factorization $c=u \cdot v$,
\item a surjective linear map $A' \colon V^u\to \Fq^r$, and
\item a surjective linear map $B' \colon V^v\to \Fq^s$.
\end{compactitem}
\end{remark}

\begin{remark}
One might hope to exploit special properties of the coefficients $M$
that appear in \eqref{eq:defM}, \eqref{eq:def-qM} to find
combinatorial connections.  For example, by inclusion-exclusion one
has that $M^{m}_{r_1,\ldots,r_k}$ counts tuples $(T_1,T_2,\ldots,T_k)$
of subsets $T_i \subseteq [m]$ such that $T_1 \cap T_2 \cap \cdots
\cap T_k = \varnothing$ and  $|T_i|=r_i$.  The obvious analogy is to
count tuples $(W_1, \ldots, W_k)$ of subspaces $W_i \subseteq \F_q^m$
such that $W_1 \cap \cdots \cap W_k = \{ 0 \}$ and $\dim W_i = r_i$.  The number of such tuples of subspaces may be computed by M\"obius inversion, but unfortunately it is
\[
\sum_{a = 0}^{\min(r_i)} (-1)^a q^{\binom{a}{2}}\qbin{m}{a}{q} \prod_{i = 1}^k \qbin{m - a}{r_i - a}{q},
\]
which is off by a power of $q$ in the summand from $M^m_{r_1, \ldots, r_k}(q)$.  Thus, we currently lack a more concrete connection than the obvious $\lim_{q \to 1}M^m_{r_1, \ldots, r_k}(q) = M^m_{r_1, \ldots, r_k}$.
\end{remark}

\begin{remark}
Note that the derivation in Remark~\ref{rmk: genus 0 many factors} is
very involved, while the corresponding derivation in the symmetric group
using the combinatorial interpretation of $M^n_{r_1,\ldots,r_k}$ is 
straightforward. One expected benefit of a combinatorial
interpretation of $M^n_{r_1,\ldots,r_k}(q)$ would be a simpler
derivation of \eqref{eq:genus 0}.
\end{remark}

\subsection{Connection with random matrices}
If $A$ is a square matrix, define (in an abuse of notation) the
power-sum symmetric function $p_k$ on $A$ by
$p_k(A):=\Tr(A^k)$. Hanlon--Stanley--Stembridge showed \cite[\S
2]{HSS} that for $x\leq y$ positive integers, the left side of \eqref{eq:Sn-twofactors-cycles} can be written as 
\begin{equation} \label{eq:expectationSn-cycles}
\sum_{r,s \geq 0} a_{r,s} \cdot x^r y^s = \mathbb{E}_U(p_n(AUBU^*))=\mathbb{E}_V(p_n(VV^*)),
\end{equation}
where the first expectation is over $y\times y$ random matrices $U$
with independent standard normal complex entries, $A=I_x\oplus {\bf 0}_{y-x}$
and $B=I_y$, and the second expectation is over $x\times y$ random matrices $V$
with independent standard normal complex entries.  Similarly, the left side of \eqref{q=1 refined generating function}
 can be written as 
\begin{equation} \label{eq:expectationSn-type}
\sum_{\lambda,\mu \vdash n} a_{\lambda,\mu} \cdot p_{\lambda}(a_1,\ldots,a_m)
p_{\mu}(b_1,\ldots,b_m) = \mathbb{E}_U(p_{n}(AUBU^*)),
\end{equation}
where the expectation is over random $m\times m$ matrices $U$ with
independent standard normal complex entries and $A,B$ are arbitrary
fixed $m\times m$ Hermitian complex matrices with eigenvalues
$a_1,\ldots,a_m$ and $b_1,\ldots,b_m$, respectively.  In \cite{GouldenJacksonRM}, Goulden and Jackson gave a combinatorial proof of these equations in terms of
factorizations and coloring of cycles. Is there an analogue of
\eqref{eq:expectationSn-cycles} for $\GL_n(\F_q)$ expressing the left side
of \eqref{eq: two factor theorem} as an expectation over random matrices?

\subsection{Other asymptotic questions}

As usual, define $F(x, y) = \sum_{r, s\geq 0} a_{r,s}(q) x^r y^s$.  In light of Remark~\ref{rmk: total number of factorizations}, we may view 
$F(x, x)/|G| = \frac{1}{|G|} \sum_{r, s\geq 0} a_{r,s}(q) x^{r + s}$  as encoding the probability distribution of the genus of a factorization chosen uniformly at random.  In particular, we have that
\[
\left. \frac{d}{dx} \left(\frac{F(x, x)}{|G|}\right) \right|_{x = 1} = \frac{1}{|G|} \sum_{r, s \geq 0} (r + s) \cdot a_{r, s}(q)
\]
is exactly equal to $n$ minus the expected genus of a random factorization.  Since $\left. (x; q^{-1})_t \right|_{x = 1} = 0$ for $t > 0$, we have by Theorem~\ref{thm: two factors} that
\begin{align}
\notag
\left. \frac{d}{dx}\left( \frac{F(x, x)}{|G|} \right)\right|_{x = 1} & = 2 \sum_{t = 1}^n b_{t, 0}(q) \cdot \left. \frac{d}{dx}\frac{(x; q^{-1})_t}{(q; q)_t}\right|_{x = 1} \\
\notag
& = 2 \sum_{t = 1}^n - \frac{(q^{-1}; q^{-1})_{t - 1}}{(q; q)_t} \\
\label{eq: expectation sum}
& = 2 \sum_{t = 1}^n \frac{(-1)^t}{q^{\binom{t}{2}} (1 - q^t)}.
\end{align}
Thus, the expected genus of a random factorization of a regular elliptic element in $\GL_n(\Fq)$ into two factors is exactly
\[
n -  2 \sum_{t = 1}^n \frac{(-1)^t}{q^{\binom{t}{2}} (1 - q^t)}.
\]
(Similar calculations could be made for the case of more factors.)  

Since the sum \eqref{eq: expectation sum} converges as $n \to \infty$,  
the vast majority of factorizations of a regular elliptic element $c \in \GL_n(\Fq)$ into two factors must have large genus.  Unfortunately, the techniques used to prove Theorem~\ref{thm:growth_rate} are not sufficient to compute asymptotics for the number of genus-$g$ factorizations if $g$ grows with $n$.  This leads to several natural questions.
\begin{question}
What is the asymptotic growth rate of the number of genus-$g$ factorizations of a regular elliptic element in $\GL_n(\Fq)$ into two factors if $g$ grows with $n$? For example, if $g = \alpha n$ for $\alpha \in (0, 1)$?
\end{question}
\begin{question}
Can one compute the limiting distribution of the genus of a random factorization of a regular elliptic element $c \in \GL_n(\Fq)$ into two factors when $n$ is large?  That is, choose $u$ uniformly at random in $\GL_n(\Fq)$ and let $v = u^{-1}c$; what is the distribution of the genus of the factorization $c = u \cdot v$?  As a first step, can one compute any higher moments of this distribution?
\end{question}

\subsection{Connection with supercharacters}

Proposition~\ref{character values prop} shows that certain characters have very simple values on the sum of a large number of elements.  This behavior is one characteristic of supercharacter theories (see, e.g., \cite{DiaconisIsaacs, DiaconisThiem}).
\begin{question}
Does the grouping of elements by fixed space dimension, as in Proposition~\ref{character values prop}, correspond to a (very coarse) supercharacter theory for $\GL_n(\F_q)$?
\end{question}

\appendix
\section{Calculation of generating functions for characters}
\label{calculation of f}

In this section we derive formulas for the character generating functions
\[
f_V(x) = \sum_{r = 0}^n \normchi^V(z_r) \cdot x^r
\]
required for the proof of our main theorems.

\begin{recap f}
If $V = (U, \lambda)$ for $U \neq \1$ we have
\begin{equation*}
\tag{\ref{easy f}}
f_{U, \lambda}(x) = |G| \cdot \frac{(x; q^{-1})_n}{(q; q)_n},
\end{equation*}
while if $V = (\1, \hook{d}{n})$ we have
\begin{equation*}
\tag{\ref{expressed in new basis}}
f_{\1, \hook{d}{n}}(x) = |G| \cdot  \left(\frac{(x; q^{-1})_n}{(q; q)_n} + q^{-d} \cdot \sum_{m = d}^{n - 1} \frac{[m]!_q \cdot [n - d - 1]!_q}{[m - d]!_q \cdot [n - 1]!_q} \cdot \frac{(x; q^{-1})_m}{(q; q)_m}\right).
\end{equation*}
\end{recap f}

\begin{proof}
To prove \eqref{easy f}, we substitute the value of $\normchi^{U, \lambda}(z_r)$ from Proposition~\ref{character values prop}(i) into the definition of $f_V$ and apply the $q$-binomial theorem \eqref{q-binomial 1} directly to get
\[
f_{U, \lambda}(x) = \sum_{r = 0}^n (-1)^{n - r} q^{\binom{n - r}{2}} \qbin{n}{r}{q} x^{r} 
= |G| \cdot \frac{(x; q^{-1})_n}{(q; q)_n},
\]
as desired.

To prove \eqref{expressed in new basis}, we substitute from Proposition~\ref{character values prop}(ii) to get the monstrous equation
\begin{multline*}
f_{\1, \hook{d}{n}}(x) = \sum_{r = 0}^n x^{r} \cdot \Big((-1)^{n - r} q^{\binom{n - r}{2}} \Big(\qbin{n}{r}{q} + \\
\frac{(1-q)[n]_q}{[r]!_q} \cdot \sum_{j = 1}^{n - \max(r, d)} q^{jr - d} \cdot \frac{[n - j]!_q }{[n - r - j]!_q}\cdot(q^{n-d-j+1};q)_{j-1}\Big)\Big).
\end{multline*}
By \eqref{q-binomial 2}, the coefficient of $\frac{(x; q^{-1})_m}{(q; q)_m}$ in the right side after changing bases is
\begin{multline}
\label{change of basis coefficient}
c_m := \sum_{r = m}^{n} (-1)^m q^{\binom{m}{2}}(q^{r}; q^{-1})_m \cdot (-1)^{n - r} q^{\binom{n - r}{2}} \Big(\qbin{n}{r}{q} + \\
\frac{1-q^n}{[r]!_q} \cdot \sum_{j = 1}^{n - \max(r, d)} q^{jr - d} \cdot \frac{[n - j]!_q }{[n - r - j]!_q}\cdot(q^{n-d-j+1};q)_{j-1}\Big).
\end{multline}
We manipulate this expression, using the following special case of the $q$-binomial theorem \eqref{q-binomial 1}:
\[
\sum_{i = 0}^k (-1)^i q^{\binom{i}{2}} \qbin{k}{i}{q} = \begin{cases} 1 & k = 0, \\ 0 & \textrm{otherwise}.\end{cases}
\]
This yields
\begin{align*}
c_m & =
(-1)^m q^{\binom{m}{2}} \Big( 
\sum_{r = m}^{n} (-1)^{n - r} q^{\binom{n - r}{2}} \qbin{n}{r}{q}  (q^{r}; q^{-1})_m
+ {} \\
& \qquad  
(1-q^n) \cdot \sum_{r = m}^{n - 1}\frac{(-1)^{n - r} q^{\binom{n - r}{2}}(q^{r}; q^{-1})_m}{[r]!_q}  \sum_{j = 1}^{n - \max(r, d)}  q^{jr - d} \cdot \frac{[n - j]!_q }{[n - r - j]!_q}\cdot(q^{n-d-j+1};q)_{j-1} \Big)\\
& =
(-1)^m q^{\binom{m}{2}} \Big( 
(1 - q)^m \sum_{r = m}^{n} (-1)^{n - r} q^{\binom{n - r}{2}}\frac{[n]!_q}{[n - r]!_q [r - m]!_q} 
+ {} \\
& \qquad 
(1-q^n) (1 - q)^m\cdot \sum_{r = m}^{n - 1}\frac{(-1)^{n - r} q^{\binom{n - r}{2}}}{[r - m]!_q}  \sum_{j = 1}^{n - \max(r, d)}  q^{jr - d} \cdot \frac{[n - j]!_q }{[n - r - j]!_q}\cdot(q^{n-d-j+1};q)_{j-1} \Big)
\end{align*}
\begin{align*}
\phantom{c_m} & =
(q - 1)^m q^{\binom{m}{2}} \Big( 
\frac{[n]!_q}{[n - m]!_q}\sum_{r = m}^{n} (-1)^{n - r} q^{\binom{n - r}{2}} \qbin{n - m}{n - r}{q}
+ {} \\
& \qquad 
(1-q^n)\sum_{j = 1}^{n - \max(m, d)} q^{jn - d} [n - j]!_q (q^{n - d- j + 1}; q)_{j - 1}
\sum_{r = m}^{n - j}\frac{ (-1)^{n - r} q^{\binom{n - r}{2} - (n - r)j}}{[n - r - j]!_q[r - m]!_q} \Big)\\
& = 
(q-1)^m q^{\binom{m}{2}} \Big( 
\frac{[n]!_q}{[n - m]!_q}\delta_{m, n}
+ {}\\
& \qquad 
(1-q^n)
\hspace{-.1in}\sum_{j = 1}^{n - \max(m, d)}\hspace{-.1in}
 q^{jn - d - \binom{j + 1}{2}} \frac{[n - j]!_q (q^{n - d- j + 1}; q)_{j - 1}}{[n - m - j]!_q}
\sum_{r = m}^{n - j} (-1)^{n - r} q^{\binom{n - r - j}{2}}\qbin{n - m - j}{n - r - j}{q} \Big)\\
& = 
(q-1)^m q^{\binom{m}{2}} \Big( 
[n]!_q \delta_{m, n}
+
(1-q^n)
\hspace{-.1in}\sum_{j = 1}^{n - \max(m, d)}\hspace{-.1in}
(-1)^j q^{jn - d - \binom{j + 1}{2}} \frac{[n - j]!_q (q^{n - d- j + 1}; q)_{j - 1}}{[n - m - j]!_q} \delta_{n - j, m}\Big).
\end{align*}
We now consider cases: if $d > m$ then this evaluates to $0$ (we have $d < n$ so $\delta_{m, n} = 0$, while every term in the second sum satisfies $j \leq n - d < n - m$ and so $\delta_{n - j, m} = 0$).  If $m = n$ then it evaluates to $(q - 1)^n q^{\binom{n}{2}} [n]!_q = |G|$ (the second sum is empty).  Finally, if $d \leq m < n$ it simplifies to 
\begin{multline*}
(q-1)^m q^{\binom{m}{2}} 
(1-q^n)
(-1)^{n - m} q^{(n - m)n - d - \binom{n - m + 1}{2}} \frac{[m]!_q (q^{m - d + 1}; q)_{n - m - 1}}{[0]!_q} 
\\
=
(q - 1)^n q^{\binom{n}{2} - d} \frac{[n]_q \cdot [m]!_q \cdot [n - d - 1]!_q}{[m - d]!_q}.
\end{multline*}
Factoring out $|G|$ gives the desired result.
\end{proof}

\bibliography{AbsGL}{}

\begin{thebibliography}{HLRV11}

\bibitem[Ber12]{Bernardi}
O.~Bernardi.
\newblock An analogue of the {H}arer-{Z}agier formula for unicellular maps on
  general surfaces.
\newblock {\em Adv. in Appl. Math.}, 48(1):164--180, 2012.

\bibitem[BM13]{BernardiM}
O.~Bernardi and A.~H. Morales.
\newblock Bijections and symmetries for the factorizations of the long cycle.
\newblock {\em Adv. in Appl. Math.}, 50(5):702--722, 2013.

\bibitem[BM15]{BernardiM3}
O.~Bernardi and A.~H. Morales.
\newblock Some probabilistic trees with algebraic roots, 2015.
\newblock \verb|arXiv:1501.01135|.

\bibitem[CFF13]{ChapuyFerayFusy}
G.~Chapuy, V.~F{\'e}ray, and {\'E}.~Fusy.
\newblock A simple model of trees for unicellular maps.
\newblock {\em J. Combin. Theory Ser. A}, 120(8):2064--2092, 2013.

\bibitem[DI08]{DiaconisIsaacs}
P.~Diaconis and I.~M. Isaacs.
\newblock Supercharacters and superclasses for algebra groups.
\newblock {\em Trans. Amer. Math. Soc.}, 360(5):2359--2392, 2008.

\bibitem[DT09]{DiaconisThiem}
P.~Diaconis and N.~Thiem.
\newblock Supercharacter formulas for pattern groups.
\newblock {\em Trans. Amer. Math. Soc.}, 361(7):3501--3533, 2009.

\bibitem[Fro68]{frobenius}
F.~G. Frobenius.
\newblock Uber gruppencharacktere.
\newblock In {\em Gesammelte {A}bhandlungen. {B}\"ande {III}}, Herausgegeben
  von J.-P. Serre. Springer-Verlag, Berlin-New York, 1968.

\bibitem[FS15]{FulmanStanton}
J.~Fulman and D.~Stanton.
\newblock On the distribution of the number of fixed vectors for the finite
  classical groups.
\newblock \verb|arXiv:1505.06383|, 2015.

\bibitem[Ful99]{Fulman}
J.~Fulman.
\newblock A probabilistic approach toward conjugacy classes in the finite
  general linear and unitary groups.
\newblock {\em J. Algebra}, 212, 1999.

\bibitem[GJ95]{GouldenJacksonRM}
I.~P. Goulden and D.~M. Jackson.
\newblock Combinatorial constructions for integrals over normally distributed
  random matrices.
\newblock {\em Proc. Amer. Math. Soc.}, 123(4):995--1003, 1995.

\bibitem[GJ14]{GouldenJacksonSurvey}
I.~P. Goulden and D.~M. Jackson.
\newblock Transitive factorizations of permutations and geometry, 2014.
\newblock \verb|arXiv:1407.7568|.

\bibitem[GR04]{GasperRahman}
G.~Gasper and M.~Rahman.
\newblock {\em Basic hypergeometric series}, volume~96 of {\em Encyclopedia of
  Mathematics and its Applications}.
\newblock Cambridge University Press, Cambridge, second edition, 2004.

\bibitem[GR15]{GrinbergReiner}
D.~Grinberg and V.~Reiner.
\newblock Hopf algebras in combinatorics, 2015.
\newblock \verb|arXiv:1409.8356v3|.

\bibitem[Gre55]{Green}
J.~A. Green.
\newblock The characters of the finite general linear groups.
\newblock {\em Trans. Amer. Math. Soc.}, 80:402--447, 1955.

\bibitem[GS98]{GoupilSchaeffer}
A.~Goupil and G.~Schaeffer.
\newblock Factoring $n$-cycles and counting maps of given genus.
\newblock {\em European J. Combin.}, 19:819--834, 1998.

\bibitem[HLR15]{HLR}
J.~Huang, J.~B. Lewis, and V.~Reiner.
\newblock Absolute order in general linear groups, 2015.
\newblock \verb|arXiv:1506.03332|.

\bibitem[HLRV11]{HaLeR-V}
T.~Hausel, E.~Letellier, and F.~Rodriguez-Villegas.
\newblock Arithmetic harmonic analysis on character and quiver varieties.
\newblock {\em Duke Math. J.}, 160:323--400, 2011.

\bibitem[HSS92]{HSS}
P.~J. Hanlon, R.~P. Stanley, and J.~R. Stembridge.
\newblock Some combinatorial aspects of the spectra of normally distributed
  random matrices.
\newblock In {\em Hypergeometric functions on domains of positivity, {J}ack
  polynomials, and applications ({T}ampa, {FL}, 1991)}, volume 138 of {\em
  Contemp. Math.}, pages 151--174. Amer. Math. Soc., Providence, RI, 1992.

\bibitem[HZ86]{HarerZagier}
J.~Harer and D.~Zagier.
\newblock The {E}uler characteristic of the moduli space of curves.
\newblock {\em Invent. Math.}, 85(3):457--485, 1986.

\bibitem[Jac87]{Jackson1}
D.~M. Jackson.
\newblock Counting cycles in permutations by group characters, with an
  application to a topological problem.
\newblock {\em Trans. Amer. Math. Soc.}, 299(2), 1987.

\bibitem[Jac88]{Jackson2}
D.~M. Jackson.
\newblock Some combinatorial problems associated with products of conjugacy
  classes of the symmetric group.
\newblock {\em J. Combin. Theory Ser. A}, 49(2), 1988.

\bibitem[LRS14]{LRS}
J.~B. Lewis, V.~Reiner, and D.~Stanton.
\newblock Reflection factorizations of {S}inger cycles.
\newblock {\em J. Algebraic Combin.}, 40(3):663--691, 2014.

\bibitem[LZ04]{LandoZvonkin}
S.~K. Lando and A.~K. Zvonkin.
\newblock {\em Graphs on surfaces and their applications}, volume 141 of {\em
  Encyclopaedia of Mathematical Sciences}.
\newblock Springer-Verlag, Berlin, 2004.

\bibitem[MV13]{MV}
A.~H. Morales and E.~A. Vassilieva.
\newblock Direct bijective computation of the generating series for 2 and
  3-connection coefficients of the symmetric group.
\newblock {\em Electron. J. Combin.}, 20(2):Paper 6, 2013.

\bibitem[RSW04]{CSP}
V.~Reiner, D.~Stanton, and D.~White.
\newblock The cyclic sieving phenomenon.
\newblock {\em J. Combin. Theory Ser. A}, 108(1):17--50, 2004.

\bibitem[RSW06]{ReinerStantonWebb}
V.~Reiner, D.~Stanton, and P.~Webb.
\newblock Springer's regular elements over arbitrary fields.
\newblock {\em Math. Proc. Cambridge Philos. Soc.}, 141(2):209--229, 2006.

\bibitem[Sta81]{Sta81}
R.~P. Stanley.
\newblock Factorization of permutations into $n$-cycles.
\newblock {\em Discrete Math.}, 37(2-3):255--262, 1981.

\bibitem[Sta12]{ec1}
R.~P. Stanley.
\newblock {\em Enumerative combinatorics. {V}olume 1}, volume~49 of {\em
  Cambridge Studies in Advanced Mathematics}.
\newblock Cambridge University Press, Cambridge, second edition, 2012.

\bibitem[Ste51]{Steinberg}
R.~Steinberg.
\newblock A geometric approach to the representations of the full linear group
  over a {G}alois field.
\newblock {\em Trans. Amer. Math. Soc.}, 71:274--282, 1951.

\bibitem[SV08]{SchaefferVassilieva}
G.~Schaeffer and E.~A. Vassilieva.
\newblock A bijective proof of {J}ackson's formula for the number of
  factorizations of a cycle.
\newblock {\em J. Combin. Theory, Ser. A}, 115(6):903--924, 2008.

\bibitem[Zel81]{Zelevinsky}
A.~V. Zelevinsky.
\newblock {\em Representations of finite classical groups, a {H}opf algebra
  approach}, volume 869 of {\em Lecture Notes in Mathematics}.
\newblock Springer-Verlag, Berlin-New York, 1981.

\end{thebibliography}
\bibliographystyle{alpha}

\end{document}